\documentclass[11pt]{article}

\usepackage[a4paper,margin=1in]{geometry}
\usepackage[T1]{fontenc}
\usepackage{lmodern}
\usepackage{amsmath,amssymb,amsthm,mathtools}
\usepackage{microtype}
\usepackage[hidelinks]{hyperref}

\newtheorem{theorem}{Theorem}[section]
\newtheorem{lemma}[theorem]{Lemma}
\newtheorem{proposition}[theorem]{Proposition}
\newtheorem{corollary}[theorem]{Corollary}
\newtheorem{conjecture}[theorem]{Conjecture}
\theoremstyle{definition}
\newtheorem{definition}[theorem]{Definition}
\theoremstyle{remark}

\newtheorem*{remark*}{Remark}

\newcommand{\Z}{\mathbb Z}
\newcommand{\one}{\mathbf 1}
\newcommand{\E}{\mathsf E}

\title{Cyclic Incidence Orderings of Complete Graphs\\
and \(3\)-Uniform Hypergraphs}
\author{%
  Junyu Zhou\\[2pt]
  \small\href{mailto:helloworld@junyu33.me}{\texttt{helloworld@junyu33.me}}
}
\date{}
\hypersetup{
  pdftitle={Cyclic Incidence Orderings of Complete Graphs and 3-Uniform Hypergraphs},
  pdfauthor={Junyu Zhou},
  pdfsubject={Classifications of cyclic orderings with shift-equivalent incidence rows},
  pdfkeywords={cyclic orderings, incidence sequences, complete hypergraphs,
    cyclic shifts, group factorizations, combinatorial designs}
}

\begin{document}
\maketitle

\begin{abstract}
We study cyclic orderings of all edges of a complete \(k\)-uniform
hypergraph on \(n\) vertices in which the binary incidence sequences
of the vertices are cyclic shifts of a common word. The shifts are
chosen independently, with no prescribed action on the vertices.
For \(2\leq k<n\), coprimality \(\gcd(n,k)=1\) is known
to suffice even when consecutive edges must differ by a single
vertex exchange. We recall a short orbit construction and prove
the converse for the first two nontrivial uniformities without
any adjacency requirement.
For \(k=2\), an ordering exists exactly when \(n=2\) or \(n\) is odd;
for \(k=3\), exactly when \(n=3\) or \(3\nmid n\).
The necessity proofs use reflected convolution identities and
pair-intersection counts to constrain the vertex shifts to a
torsion coset. For triples, multiplicity-preserving dilation and
conditional prime-power capacity bounds complete the argument.
\end{abstract}

\medskip
\noindent\textbf{Keywords.}
cyclic orderings; incidence sequences; complete hypergraphs;
cyclic shifts; group factorizations; combinatorial designs.

\medskip
\noindent\textbf{2020 Mathematics Subject Classification.}
Primary 05C65; Secondary 05A05, 05B20.

\section{Introduction}

Consider a research group with \(n\) students, exactly \(k\) of whom
must attend each meeting. Over a complete cycle, every \(k\)-element
set of students is to meet exactly once, giving
\(N=\binom nk\) meetings. For each student \(v\), record the cyclic
binary attendance sequence
\[
  x_v(t)=
  \begin{cases}
    1,&\text{if \(v\) attends meeting \(t\)},\\
    0,&\text{otherwise},
  \end{cases}
  \qquad t\in\Z_N.
\]
Call the schedule \emph{fair} if these sequences are cyclic shifts
of one another. Thus each student follows the same attendance
pattern, with an individually chosen starting point.

For example, four students have six possible pair meetings:
\[
  \{1,2\},\ \{1,3\},\ \{1,4\},\ \{2,3\},\ \{2,4\},\ \{3,4\}.
\]
Can these pairs be placed in a cyclic order so that the four
attendance sequences differ only by shifts? The answer is no, as
the graph classification below shows.

For integers \(n\geq k\geq1\), write \([n]=\{1,\ldots,n\}\)
and \(\Z_N=\Z/N\Z\).

\begin{definition}\label{def:ordering}
A \emph{cyclic incidence ordering} is a bijection
\[
  e:\Z_N\longrightarrow\binom{[n]}k,\qquad N=\binom nk,
\]
for which there exist a set \(D\subseteq\Z_N\) and shifts
\(s_v\in\Z_N\), \(v\in[n]\), such that
\[
  \{t\in\Z_N:v\in e(t)\}=D+s_v .
\]
We write \(\E(n,k)\) when such an ordering exists.
\end{definition}

The condition in Definition~\ref{def:ordering} concerns each incidence row
separately. It assumes neither a common shift step between
consecutive vertices nor a cyclic action on the vertex set. In
particular, it imposes no global rotational or dihedral symmetry
and no predetermined group action on the students.

\begin{remark*}[Complement symmetry]
For \(1\leq k<n\), one has \(\E(n,k)\iff\E(n,n-k)\).
Indeed, the complementary ordering \(e^c(t)=[n]\setminus e(t)\)
has attendance supports
\(\Z_N\setminus(D+s_v)=(\Z_N\setminus D)+s_v\),
and complementation is an involution. Thus the graph and triple
classifications below also determine the cases \(k=n-2\) and
\(k=n-3\), respectively, whenever these uniformities are positive.
\end{remark*}

Viewing each meeting by its incidence vector identifies the meetings
with the constant-weight layer
\(B^k(n)=\{x\in\{0,1\}^n:\sum_i x_i=k\}\).
Thus \(\E(n,k)\) asks for a cyclic array containing every word of
\(B^k(n)\) once as a row, whose coordinate columns are cyclic shifts
of a common column. This is the \emph{single-track property} in the
single-track Gray codes introduced by Hiltgen, Paterson, and
Brandestini \cite{HPB} and investigated structurally by Schwartz
and Etzion \cite{SE}. Their codes also require consecutive rows
to differ in one bit; we impose no adjacency condition.

Ruskey, Sawada, and Williams \cite[Section~7, Problem~3]{RSW}
observed that their fixed-weight de Bruijn sequences give
single-track orders of the shorthand language
\(B^k_{k-1}(n-1)\), consisting of words of length \(n-1\) and
weight \(k-1\) or \(k\). They asked which other sets of binary
strings admit single-track orders. The full layer \(B^k(n)\) is
a natural specialization of that question. Their construction
does not assert the single-track property after the omitted final
bit is restored. Universal cycles for \(k\)-subsets \cite{CDG}
encode subsets by overlapping length-\(k\) windows over \([n]\);
this is a different requirement from cyclic shifts of incidence
columns.

En Gad et al. \cite[Section~IV.B]{EGLSB} give constant-weight
single-track constructions that cover full layers for some
parameters, with each step changing \(10\) to \(01\) in cyclically
adjacent coordinates. For \(2\leq k<n\), their first-moment
argument also forces \(\gcd(n,k)=1\) for cyclic codes covering
\(B^k(n)\) when \(k\) is prime
\cite[Theorems~18--19]{EGLSB}. This argument uses the prescribed
transitions. Gregor, Merino, and M\"utze
\cite[Theorem~4.5]{GMM} construct a single-track Hamilton cycle
in the Johnson graph on \(B^k(n)\) whenever \(\gcd(n,k)=1\).
Consecutive words there differ by exchanging a \(0\) and a \(1\),
so their result implies the sufficiency below with an additional
adjacency requirement. We recall the short orbit-development
proof, which requires no adjacency.

\begin{proposition}[Translation-orbit construction]\label{prop:construction}
Let \(2\leq k<n\) and \(\gcd(n,k)=1\).
Then \(\E(n,k)\) holds.
\end{proposition}

\begin{proof}
Identify the vertices with \(\Z_n\), acting on \(k\)-subsets by
translation. If a \(k\)-subset \(B\) has stabilizer \(H\), then
\(|H|\mid n\). Moreover, \(B\) is a union of \(H\)-cosets, so
\(|H|\mid k\). Hence \(H\) is trivial and every orbit has length \(n\).

Choose representatives \(B_0,\ldots,B_{q-1}\) for the orbits, where
\(q=\binom nk/n\). Define
\[
  e(aq+j)=B_j+a
  \qquad(0\leq a<n,\ 0\leq j<q).
\]
This lists every \(k\)-subset once. With time interpreted modulo
\(nq\), it satisfies \(e(t+q)=e(t)+1\). Consequently, if \(D\)
is the attendance support of vertex \(0\), then the support of
vertex \(v\) is \(D+vq\).
\end{proof}

The ordering in this construction has a specified translation
symmetry. Its use for existence places no symmetry assumption on
an arbitrary ordering.

Thus coprimality is always sufficient for \(2\leq k<n\), and
the substantive question is whether it is also necessary.
In this paper we answer this question affirmatively for the
first two nontrivial uniformities, \(k=2\) and \(k=3\).

\begin{theorem}[Graphs]\label{thm:graph}
For \(n\geq2\), \(\E(n,2)\) holds if and only if \(n=2\) or \(n\) is odd.
\end{theorem}

\begin{theorem}[Triples]\label{thm:triple}
For \(n\geq3\), \(\E(n,3)\) holds if and only if \(n=3\) or \(3\nmid n\).
\end{theorem}

For necessity, we first derive convolution and pair-intersection
identities valid for every \(2\leq k<n\). Reflecting a convolution
factor preserves balance, and a count of local coordinates imposes
an additive constraint on the vertex shifts. This already settles
the graph case. For triples, the resulting torsion reduction is
followed by a bound on template points in a residue class, obtained
by first dilating the shift multiplicities and then applying a
prime-power Fourier estimate on each coset.

Together with the general construction, these two classifications
suggest that coprimality is the exact existence criterion.
We formulate this assertion as Conjecture~\ref{conj:main} in the concluding
section; its necessity in higher uniformities remains open.

\section{Universal constraints on cyclic incidence orderings}

We now develop necessary identities that hold for every \(2\leq k<n\).

For functions on a finite abelian group \(G\), use the convolution
and Fourier transform
\[
 (F*H)(t)=\sum_{y\in G}F(y)H(t-y),
 \qquad
 \widehat F(\chi)=\sum_{y\in G}F(y)\overline{\chi(y)}
 \quad(\chi\in\widehat G).
\]
Here \(\widehat G\) is the character group, \(\one_D\) is the
indicator of a set \(D\), and \(\one_G\) is the constant function
with value \(1\). Fourier transformation takes convolution to
pointwise multiplication.

\begin{lemma}[Reflected balance]\label{lem:reflection}
Let \(F,H:G\to\mathbb R\) satisfy \(F*H=r\one_G\).
For \(F^-(x)=F(-x)\) and \(H^-(x)=H(-x)\), one has
\[
  F^-*H=F*H^-=r\one_G .
\]
\end{lemma}

\begin{proof}
For real \(F\), \(\widehat{F^-}(\chi)=\overline{\widehat F(\chi)}\),
so reflection preserves its Fourier zero set. At every nontrivial
character, \(\widehat F(\chi)\widehat H(\chi)=0\), and this remains
true after reflecting either factor. At the trivial character the
two masses, and hence their product, are unchanged. Fourier
inversion proves both identities.
\end{proof}

This lemma reflects a convolution identity only. It requires no
reflection symmetry of an ordering.

Suppose that \(\E(n,k)\) holds with \(2\leq k<n\), and put
\(N=\binom nk\).
The shifts in Definition~\ref{def:ordering} are distinct: for distinct vertices
\(u,v\), there is a \(k\)-subset containing \(u\) and excluding \(v\),
so their incidence rows differ. Let \(S=\{s_v:v\in[n]\}\).
Counting meetings at a vertex, at a time, and at a pair of vertices
gives, respectively,
\begin{gather}
 |S|=n,\qquad |D|=\binom{n-1}{k-1},\qquad
 \one_S*\one_D=k\one_{\Z_N},\label{eq:balance}\\
 |(D+b)\cap(D+c)|=\binom{n-2}{k-2}
 \qquad(b,c\in S,\ b\ne c).\label{eq:paircount}
\end{gather}
These identities do not require a choice of equally spaced shifts.

\begin{lemma}[Coordinate count]\label{lem:coordinates}
Let \(S,D\subseteq G\), let \(|S|=m\), and suppose that
\(\one_S*\one_D=r\one_G\). Fix \(b\in S\), and index all pairs
\[
  (c,t)\quad\text{with}\quad
  c\in S\setminus\{b\},\qquad t\in(D+b)\cap(D+c).
\]
In the two multisets of coordinates \(x=t-b\) and \(y=t-c\),
every element of \(D\) occurs exactly \(r-1\) times.
If every intersection above has size \(\lambda\), then, with
\(T=\sum_{c\in S}c\) as a sum in \(G\),
\begin{equation}\label{eq:sumidentity}
  \lambda(T-mb)=0.
\end{equation}
\end{lemma}

\begin{proof}
For \(x\in D\), the original balance at \(b+x\) counts \(r\)
representations \(b+x=c+y\), with \(c\in S\) and \(y\in D\).
Exactly one has \(c=b\), namely \(y=x\). Thus \(x\) occurs \(r-1\)
times in the first multiset.

For a fixed \(y\in D\), write the same equation as
\(c-x=b-y\). By Lemma~\ref{lem:reflection},
\(\one_S*\one_{-D}=r\one_G\), so there are exactly \(r\)
representations of this form with \(c\in S\) and \(x\in D\).
Again the unique representation with \(c=b\) is \(x=y\).
Thus \(y\) also occurs \(r-1\) times.

Summing \(x-y=c-b\) over the indexed pairs now gives zero on the
left, since the two coordinate multisets agree. If each
intersection has size \(\lambda\), the right side is
\(\lambda\sum_{c\ne b}(c-b)=\lambda(T-mb)\).
\end{proof}

\section{The graph case}

\begin{proof}[Proof of Theorem~\ref{thm:graph}]
For \(n=2\) the unique edge gives a one-meeting cycle.
For odd \(n\geq3\), apply Proposition~\ref{prop:construction} with \(k=2\).

Conversely, suppose \(\E(n,2)\) holds and \(n>2\). Put
\(N=\binom n2\) and use \(D,S\subseteq\Z_N\) from
Equations~\eqref{eq:balance} and~\eqref{eq:paircount}. Here \(|S|=n\), the convolution
height is \(2\), and every pair of translates has intersection
size \(1\). In Lemma~\ref{lem:coordinates}, each of the two coordinate
multisets therefore contains every element of \(D\) once. Writing
\(T=\sum_{c\in S}c\), Equation~\eqref{eq:sumidentity} gives
\[
  nb=T\qquad(b\in S).
\]
Fixing \(b_0\in S\), we obtain \(n(b-b_0)=0\) for all \(b\in S\).
Thus the \(n\) distinct shifts lie in a coset of
\[
  \{x\in\Z_N:nx=0\},
\]
which has size \(\gcd(n,N)\). Hence \(n\leq\gcd(n,N)\leq n\),
so \(n\mid N\).

If \(n=2h\) is even, then \(N=h(2h-1)\) and
\[
  \gcd(n,N)=h\gcd(2,2h-1)=h=n/2,
\]
contradicting the required size of the torsion subgroup. This
excludes every even \(n>2\).
\end{proof}

\begin{corollary}[Forced cyclic symmetry]\label{cor:graphsymmetry}
Let \(n\geq3\) be odd, and put \(N=\binom n2\) and
\(q=N/n=(n-1)/2\). For every cyclic incidence ordering \(e\)
of \(\binom{[n]}2\), there is an \(n\)-cycle \(\sigma\) on
\([n]\) such that
\[
  e(t+q)=\sigma(e(t))\qquad(t\in\Z_N).
\]
\end{corollary}

\begin{proof}
The preceding proof places the \(n\) distinct shifts in a coset
of \(\{x\in\Z_N:nx=0\}\). Since \(N=nq\), this subgroup
has exactly \(n\) elements, namely \(0,q,\ldots,(n-1)q\).
Hence, for some \(b_0\),
\[
  S=b_0+\{0,q,\ldots,(n-1)q\}.
\]
Define \(\sigma\) by \(s_{\sigma(v)}=s_v+q\).
Then \(\sigma\) is an \(n\)-cycle and
\(D+s_{\sigma(v)}=D+s_v+q\). Consequently,
\(\sigma(v)\in e(t+q)\) if and only if \(v\in e(t)\),
which proves the assertion.
\end{proof}

\section{The triple case: torsion reduction}

Suppose \(\E(n,3)\) holds with \(n>3\), and put \(M=\binom n3\).
The shifts \(S\subseteq\Z_M\) and template \(D\subseteq\Z_M\)
satisfy
\[
  |S|=n,\qquad
  \one_S*\one_D=3\one_{\Z_M},\qquad
  |(D+b)\cap(D+c)|=n-2\quad(b\ne c).
\]
In Lemma~\ref{lem:coordinates} both coordinate multisets contain every
element of \(D\) twice. Therefore
\[
  (n-2)(T-nb)=0\qquad(b\in S),
  \qquad T=\sum_{c\in S}c.
\]
Subtracting the equations for two shifts gives
\begin{equation}\label{eq:tripletorsion}
  n(n-2)(b-b')=0\quad\text{in }\Z_M
  \qquad(b,b'\in S).
\end{equation}
No integer factor has been cancelled in this identity.

\begin{proposition}[Reduction to coset slices]\label{prop:reduction}
Suppose \(3\mid n\), \(n>3\), and \(\E(n,3)\) holds. Define
\[
  (L,d,c)=
  \begin{cases}
    \bigl(n(n-2)/3,\ n-2,\ (n-1)/2\bigr),&n\text{ odd},\\
    \bigl(n(n-2)/6,\ (n-2)/2,\ n-1\bigr),&n\text{ even}.
  \end{cases}
\]
Then \(M=cL\) and \(L=(n/3)d\). There exist a common set
\(A\subseteq\Z_L\), \(|A|=n\), and sets
\(D_0,\ldots,D_{c-1}\subseteq\Z_L\) such that
\begin{equation}\label{eq:slicebalance}
  \one_A*\one_{D_j}=3\one_{\Z_L},
  \qquad |D_j|=d\qquad(0\leq j<c).
\end{equation}
On slice \(j\), at coordinate \(x\in\Z_L\), the attending vertices
are exactly the shifts in \(A\cap(x-D_j)\).
\end{proposition}

\begin{proof}
The subgroup annihilated by \(n(n-2)\) in \(\Z_M\) has order
\(L=\gcd(M,n(n-2))\). If \(n\) is odd and divisible by \(3\),
write \(L_0=n(n-2)/3\). Then \(M=L_0(n-1)/2\) and
\(n(n-2)=3L_0\), while \(\gcd(3,(n-1)/2)=1\). Thus \(L=L_0\).
If \(n\) is even and divisible by \(3\), write
\(L_0=n(n-2)/6\). Now \(M=L_0(n-1)\),
\(n(n-2)=6L_0\), and \(\gcd(6,n-1)=1\), again giving \(L=L_0\).
This proves the formulas for \(L,c\), and \(d=3L/n\).

Choose \(b_0\in S\). Replace \(S\) by \(S-b_0\) and \(D\) by
\(D+b_0\), which leaves every vertex support unchanged.
By Equation~\eqref{eq:tripletorsion}, the new \(S\) lies in the subgroup
\(c\Z_M\). The map \(a\mapsto ca\) identifies \(\Z_L\) with
this subgroup, so write \(S=cA\), where \(|A|=n\).

Every time has a unique expression \(j+cx\), with
\(0\leq j<c\) and \(x\in\Z_L\). Define
\[
  D_j=\{x\in\Z_L:j+cx\in D\}.
\]
Vertex \(a\in A\) attends at time \(j+cx\) precisely when
\(x-a\in D_j\). Thus the original height-\(3\) balance gives
Equation~\eqref{eq:slicebalance} at each \(x\). Summing this identity over
\(\Z_L\) yields \(n|D_j|=3L\), and hence \(|D_j|=d\).
\end{proof}

The identification of vertices with \(A\) is common to all slices.
Each original meeting occurs on one of these slices. In particular,
every triple of vertices must still occur somewhere among them.

\section{Dilation and capacity}

We use nonnegative integer functions to retain multiplicities.
If \(F:G\to\Z_{\geq0}\) and \(u\) is a positive integer, define
the pushforward under multiplication by \(u\) by
\begin{equation}\label{eq:pushforward}
  F^{[u]}(x)=\sum_{\substack{y\in G\\uy=x}}F(y).
\end{equation}
Multiplication by \(u\) need not be injective: coincident images
are counted with their full multiplicities.

\subsection{Frobenius dilation}

Dilation arguments for translational and level tilings appear
in several forms: the onefold setting of Coven and Meyerowitz
\cite[Lemma~3.1]{CM}, the level tilings of Greenfeld and Tao
\cite[Lemma~3.1]{GT}, and the measurable abelian actions of
Greb{\'i}k et al. \cite[Theorem~1.2(i) and Remark~2.2]{GGRT}.
We use the following finite-group constant-height version,
retaining multiplicities under collisions.

\begin{lemma}[Frobenius dilation]\label{lem:dilation}
Let \(F,H:G\to\Z_{\geq0}\) satisfy
\[
  F*H=\lambda\one_G,\qquad
  \sum_{x\in G}F(x)=m,
\]
where \(\lambda\) is a positive integer. If \(p\) is prime,
\(p>\lambda\), and \(p\nmid m\), then
\[
  F^{[p]}*H=\lambda\one_G.
\]
Consequently \(F^{[Q]}*H=\lambda\one_G\) whenever every prime
factor of \(Q\) is greater than \(\lambda\) and does not divide \(m\).
\end{lemma}

\begin{proof}
In the group ring over \(\mathbb F_p\), the Frobenius identity gives
\[
  F^{[p]}\equiv F^{*p}\pmod p.
\]
Indeed, raising a sum to the \(p\)-th power multiplies its group
indices by \(p\), while the integer coefficients satisfy
\(F(y)^p\equiv F(y)\pmod p\). Hence
\begin{align*}
  F^{[p]}*H
  &\equiv F^{*(p-1)}*(F*H)\\
  &=\lambda m^{p-1}\one_G
   \equiv\lambda\one_G\pmod p.
\end{align*}
Every coefficient on the left is a nonnegative integer congruent
to \(\lambda\) modulo \(p\). Since \(0<\lambda<p\), it is at least
\(\lambda\). The pushforward preserves total mass, so
\[
  \sum_{x\in G}(F^{[p]}*H)(x)
  =m\sum_{x\in G}H(x)=\lambda|G|.
\]
All coefficients must therefore equal \(\lambda\).
Iterating this argument proves the assertion for \(Q\), since
the mass of the first factor remains \(m\) and successive
pushforwards compose.
\end{proof}

\subsection{Prime-power capacity}

A nonconstant Fourier level of a function on \(\Z_{p^a}\) is
\emph{occupied} if the Fourier transform is nonzero at some
character of exact order \(p^b\), for the corresponding
\(b\in\{1,\ldots,a\}\).

\begin{lemma}[Nonnegative Fourier levels]\label{lem:levels}
Let \(p\) be prime, \(a\geq1\), and
\(w:\Z_{p^a}\to\mathbb R_{\geq0}\) have mean \(\mu\).
If at most \(s\) nonconstant Fourier levels are occupied, then
\[
  \max_x w(x)\leq p^s\mu.
\]
\end{lemma}

\begin{proof}
For \(0\leq b\leq a\), define a function on \(\Z_{p^b}\) by
\[
  w_b(x)=\frac{1}{p^{a-b}}
         \sum_{\substack{y\in\Z_{p^a}\\y\equiv x\pmod{p^b}}}w(y).
\]
Regard it also as a function on \(\Z_{p^a}\) by reduction
modulo \(p^b\). Then \(w_0=\mu\) and \(w_a=w\).
The averaging defining \(w_b\) retains exactly the Fourier
characters whose orders divide \(p^b\). Consequently
\(w_b-w_{b-1}\) consists of the exact-order \(p^b\) level.
If that level is unoccupied, \(w_b=w_{b-1}\).

In every case, a value of \(w_{b-1}\) is the mean of the
\(p\) nonnegative values of \(w_b\) above it:
\[
  w_{b-1}(x)=\frac1p\sum_{j=0}^{p-1}w_b(x+jp^{b-1}).
\]
Thus each child value is at most \(p\) times its parent value.
Starting with \(w_0=\mu\), at most \(s\) levels can incur this
factor \(p\), giving the bound.
\end{proof}

\begin{proposition}[Prime-power capacity]\label{prop:primecapacity}
Let \(F,H:\Z_L\to\Z_{\geq0}\) satisfy
\[
  F*H=\lambda\one_{\Z_L},\qquad
  \sum F=m,\qquad \sum H=d,
\]
where \(\lambda>0\). If \(p^a\mid L\), with \(p\) prime and
\(a\geq1\), then for every \(x\in\Z_{p^a}\),
\begin{equation}\label{eq:primecapacity}
  \sum_{\substack{y\in\Z_L\\y\equiv x\pmod{p^a}}}H(y)
  \leq
  \left\lfloor\frac{p^{\min(a,v_p(m))}d}{p^a}\right\rfloor .
\end{equation}
Here \(v_p(m)\) is the exponent of \(p\) in \(m\).
\end{proposition}

\begin{proof}
Put \(q=p^a\), and let \(\alpha,w\) be the pushforwards of \(F,H\)
under reduction modulo \(q\). Projection commutes with convolution.
Each residue has \(L/q\) preimages, so
\begin{equation}\label{eq:projectedbalance}
  \alpha*w=\lambda(L/q)\one_{\Z_q},\qquad
  \sum\alpha=m,\qquad \sum w=d.
\end{equation}
At a nontrivial character \(\chi\), nonzero \(\widehat w(\chi)\)
therefore forces \(\widehat\alpha(\chi)=0\).

Consider the integer polynomial
\[
  P(X)=\sum_{x=0}^{q-1}\alpha(x)X^x.
\]
If the exact-order \(p^b\) level of \(w\) is occupied, then \(P\)
vanishes at a primitive \(p^b\)-th root of unity. Hence the
cyclotomic polynomial \(\Phi_{p^b}\) divides \(P\).
For \(s\) distinct occupied levels, the corresponding monic
irreducible cyclotomic polynomials are pairwise coprime in
\(\mathbb Q[X]\). Their product therefore divides \(P\) in
\(\Z[X]\) by Gauss's lemma. Evaluation at \(1\), using
\(\Phi_{p^b}(1)=p\), shows that
\[
  p^s\mid P(1)=m.
\]
Thus \(s\leq\min(a,v_p(m))\).

The function \(w\) is nonnegative and has mean \(d/q\).
Lemma~\ref{lem:levels} gives
\(\max w\leq p^{\min(a,v_p(m))}d/q\).
Since \(w\) is integer-valued, this proves
Equation~\eqref{eq:primecapacity}.
\end{proof}

\subsection{Dilation followed by conditioning}

\begin{proposition}[Conditional capacity]\label{prop:mixedcapacity}
Let \(F,H:\Z_L\to\Z_{\geq0}\) satisfy
\[
 F*H=\lambda\one_{\Z_L},\qquad
 \sum F=m,\qquad \sum H=d,
\]
with positive integer \(\lambda\). Suppose \(q=Qp^a\mid L\),
where \(p\) is prime, \(a\geq0\), \(\gcd(Q,p)=1\), and every
prime divisor of \(Q\) is greater than \(\lambda\) and coprime
to \(m\). Then for every \(x\in\Z_q\),
\begin{equation}\label{eq:mixedcapacity}
  \sum_{\substack{y\in\Z_L\\y\equiv x\pmod q}}H(y)
  \leq
  \left\lfloor\frac{\gcd(q,m)d}{q}\right\rfloor.
\end{equation}
For \(a=0\), the conclusion is the exact equality \(d/Q\)
for each residue modulo \(Q\).
\end{proposition}

\begin{proof}
First apply Lemma~\ref{lem:dilation} to obtain
\[
  F'=F^{[Q]},\qquad F'*H=\lambda\one_{\Z_L}.
\]
Since \(Q\mid L\), \(F'\) is supported on \(Q\Z_L\).
Set \(K=\Z_{L/Q}\), and define
\[
  F_K(y)=F'(Qy),\qquad H_z(y)=H(z+Qy)
  \quad(y\in K,\ 0\leq z<Q).
\]
For \(x\in K\), the support of \(F'\) gives
\begin{equation}\label{eq:conditionalbalance}
  \begin{aligned}
    (F_K*H_z)(x)
    &=\sum_{y\in K}F'(Qy)H\bigl(z+Q(x-y)\bigr)\\
    &=(F'*H)(z+Qx)=\lambda.
  \end{aligned}
\end{equation}
In particular, each conditional convolution has height
\(\lambda\). Moreover, \(\sum_K F_K=m\). Since \(md=\lambda L\),
summing Equation~\eqref{eq:conditionalbalance} yields
\[
  \sum_{y\in K}H_z(y)=\frac{\lambda(L/Q)}m=\frac dQ.
\]
This also proves \(Q\mid d\).

If \(a=0\), these slice masses are the asserted equality.
If \(a\geq1\), then \(p^a\mid L/Q\), and
Proposition~\ref{prop:primecapacity} applies to every pair \(F_K,H_z\).
A residue class modulo \(Qp^a\) in \(\Z_L\) is a residue class
modulo \(p^a\) within one of the slices \(z+QK\). Its mass
is therefore at most
\[
  \left\lfloor
    \frac{p^{\min(a,v_p(m))}(d/Q)}{p^a}
  \right\rfloor.
\]
Finally \(\gcd(Q,m)=1\), so
\(\gcd(q,m)=p^{\min(a,v_p(m))}\), giving
Equation~\eqref{eq:mixedcapacity}.
\end{proof}

The conditional balance \eqref{eq:conditionalbalance} is the
reason prime-power capacity applies here. The argument dilates
first and then conditions on cosets; it does not multiply bounds
obtained from independent projections.

\section{Complete classification for triples}

\begin{proof}[Proof of Theorem~\ref{thm:triple}]
The case \(n=3\) is the one-meeting cycle. If \(3\nmid n\),
Proposition~\ref{prop:construction} supplies the required ordering.

Suppose, for a contradiction, that \(3\mid n\), \(n>3\), and
\(\E(n,3)\) holds. Use \(A,D_j,L,d\) from
Proposition~\ref{prop:reduction}. For \(n=6\) that proposition would give
\(|A|=6\) in \(\Z_4\), already impossible. We may therefore
assume \(n\geq9\). Since \(L=(n/3)d\), we have \(d\mid L\).
We apply Proposition~\ref{prop:mixedcapacity} with \(q=d\),
\(F=\one_A\), \(H=\one_{D_j}\), \(m=n\), and \(\lambda=3\).

First suppose \(n\) is odd. Then \(d=n-2\) is odd,
\(3\nmid d\), and \(\gcd(n,d)=1\). Every prime divisor of \(d\)
is consequently at least \(5\) and does not divide \(n\).
Take \(Q=d\), \(p=2\), and \(a=0\) in
Proposition~\ref{prop:mixedcapacity}. On every slice,
\begin{equation}\label{eq:oddcapacity}
  |D_j\cap(x+d\Z_L)|\leq1\qquad(x\in\Z_L).
\end{equation}
There are \(n=d+2>d\) shifts in \(A\), so two distinct shifts
\(u,v\in A\) are congruent modulo \(d\). If they attended
together at a time with slice coordinate \(x\), then the two
distinct points \(x-u,x-v\) would lie in \(D_j\) and be congruent
modulo \(d\), contrary to Equation~\eqref{eq:oddcapacity}.
Thus these two vertices can never meet, contradicting the
occurrence of every triple containing them.

Now suppose \(n\) is even. Here \(n=2d+2\),
\(3\nmid d\), and \(\gcd(n,d)=\gcd(2,d)\leq2\).
Write \(d=2^aQ\), with \(Q\) odd.
Every prime factor of \(Q\) is at least \(5\) and coprime to \(n\).
Taking \(p=2\) in Proposition~\ref{prop:mixedcapacity} gives
\begin{equation}\label{eq:evencapacity}
  |D_j\cap(x+d\Z_L)|\leq\gcd(n,d)\leq2
  \qquad(x\in\Z_L)
\end{equation}
on every slice. In particular, if \(a\geq1\), then
\(n=2(d+1)\) with \(d+1\) odd, so \(v_2(n)=1\), as required
by the prime-power bound.

Since \(|A|=2d+2>2d\), some residue modulo \(d\) contains
three distinct shifts \(u,v,w\in A\). A meeting of these
three vertices would require three distinct points
\(x-u,x-v,x-w\) in one fiber of \(D_j\), contradicting
Equation~\eqref{eq:evencapacity}. The same three vertices are used on
every slice, so their required triple is absent from the
entire ordering. This completes the contradiction.
\end{proof}

\section{Concluding remarks}

The graph and triple cases exhibit different rigidity mechanisms.
For \(k=2\), reflected balance and pair-intersection counts force
the vertex shifts into a torsion coset that is too small in the
excluded cases. For \(k=3\), this torsion reduction must be combined
with multiplicity-preserving dilation and conditional capacity
bounds.

Proposition~\ref{prop:construction} recalls the known sufficiency of coprimality
for every \(2\leq k<n\), while Theorems~\ref{thm:graph} and~\ref{thm:triple} establish
its necessity for the first two nontrivial uniformities. Whether
the weak condition of independently shifted incidence rows forces
coprimality in higher uniformities remains open. The mechanism
that could impose this arithmetic rigidity beyond the first two
cases is not yet understood.

These results motivate the following conjecture. The restriction
\(k<n\) excludes the one-meeting case \(k=n\), in which a cyclic
incidence ordering always exists.

\begin{conjecture}\label{conj:main}
For integers \(2\leq k<n\),
\[
  \E(n,k)\quad\Longleftrightarrow\quad\gcd(n,k)=1.
\]
\end{conjecture}

\section*{Use of AI tools}

The author used OpenAI language-model tools, including Codex,
extensively in the mathematical development and preparation of
this paper. These tools contributed to proposing and refining
proof arguments, checking intermediate claims and complete proofs,
locating and comparing related literature, and drafting and
revising the exposition and \LaTeX{} source. The author takes
responsibility for the mathematical claims, proofs, references,
and final text.

\begingroup
\small
\urlstyle{same}
\bibliographystyle{plain}
\bibliography{refs}
\endgroup
\end{document}